\documentclass[amstex,10pt,reqno]{amsart}
\usepackage{amsmath,amsfonts,amssymb,amsthm,enumerate,hyperref,multicol}
\newtheorem{lemma}{Lemma}
\newtheorem{thm}{Theorem}

\newtheorem{corollary}{Corollary}
\newtheorem{remark}{Remark}
\numberwithin{equation}{section}
\begin{document}

\leftline{ \scriptsize \it  }
\title[]
{Generalized Baskakov Kantorovich Operators}
\maketitle

\begin{center}
{\bf P. N. Agrawal$^1$ and Meenu Goyal$^2$}
\vskip0.2in
$^{1,2}$Department of Mathematics\\
Indian Institute of Technology Roorkee\\
Roorkee-247667, India\\
\vskip0.2in

 $^1$pna\_{}iitr@yahoo.co.in
$^2$meenu.goyaliitr@yahoo.com
\end{center}
\begin{abstract}
In this paper, we construct generalized Baskakov Kantorovich operators. We
establish some direct results and then study weighted approximation,
simultaneous approximation and statistical convergence properties for these
operators. Finally, we obtain the rate of convergence for functions having a
derivative coinciding almost everywhere with a function of bounded variation
for these operators.
\end{abstract}

\section{introduction}
In \cite{KON},  for $f\in L_1[0,1]$ (the class of Lebesgue integrable
functions on $[0,1]$), Kantorovich introduced the operators
\begin{eqnarray*}
K_n(f;x)=(n+1)\sum_{k=0}^{n}p_{n,k}(x)\int_{0}^{1}\chi_{n,k}(t)f(t) dt,
\end{eqnarray*}
where $p_{n,k}(x)={n\choose k}x^{k}(1-x)^{n-k},\,\, x\in [0,1]$
is the Bernstein basis function and $\chi_{n,k}(t)$ is the characteristic
function
of the interval $\bigg[\frac{k}{n+1},\frac{k+1}{n+1}\bigg].$ Many authors studied the approximation properties of these operators.
In \cite{BTZ}, Butzer proved Voronovskaja type theorem for the Kantorovich
polynomials, then in \cite{ABL}, Abel derived the complete asymptotic expansion
and studied simultaneous approximation, rate of convergence and asymptotic
expansion for derivatives of these polynomials. Mahmudov and Sabancigil \cite{NMS} investigated a $q-$analogue of the operators $K_n$ and
studied the local and global results and a Voronovskaja type theorem for the case $0<q<1.$\\
Besides the applications of $q-$calculus in approximation theory, the approximation of functions by linear positive operators using statistical convergence is also an active area of research. The concept of statistical convergence was introduced by Fast \cite{FST}. Gadjiev and Orhan \cite{ADG} examined it for the first time for linear positive operators. After that many researchers have investigated the statistical convergence properties for several sequences and classes of linear positive operators (cf. \cite{AAD, SA, ODO, ODC, EDS, ADG, VGS, NIG,  KOK, NIM, MML, MA, ODS, EYO}).

The rate of convergence for functions of bounded variation and for functions with a derivative of bounded variation is an another important topic of research. Cheng \cite{FC} estimated the rate of convergence of Bernstein polynomials for functions of bounded variation. Guo \cite{SG} studied it for the Bernstein-Durrmeyer polynomials by using Berry Esseen theorem. Srivastava and Gupta \cite{HG1} introduced a new sequence of linear positive operators which includes several well known operators as its special cases and investigated the rate of convergence of these operators by means of the decomposition technique for functions of bounded variation. Subsequently, the rate of convergence for functions of bounded variation and for functions with a derivative of bounded variation was examined for several sequences of linear positive operators. In this direction, the significant contributions have been made by (cf. \cite{VG1, NI, HR, MA1} and \cite{HMG} etc.).\\

Several authors have proposed the Kantorovich-type modification of different
linear positive operators and studied their approximation properties related
with the operators. In \cite{VRD} and \cite{NIM}, researchers introduced the
generalizations of the q-Baskakov-Kantorovich operators and studied their
weighted statistical approximation properties.\\
In 1998, Mihesan \cite{MSN} introduced the generalized Baskakov operators $B_{n,a}^*$ defined as
\begin{eqnarray*}
 B_{n,a}^*(f;x) =\sum_{k=0}^{\infty}W_{n,k}^a(x)f\bigg(\frac{k}{n+1}\bigg),
\end{eqnarray*}
where\\
$W_{n,k}^{a}(x)=e^{\frac{-ax}{1+x}}\frac{P_{k}(n,a)}{k!}\frac{x^k}{(1+x)^{n+k}}
, P_k(n,a)=\displaystyle\sum_{i=0}^{k}{k\choose i}(n)_ia^{k-i},$
and $(n)_0=1, (n)_i=n(n+1)\cdot\cdot\cdot(n+i-1),$ for $i\geq 1.$
In \cite{ERN}, Erencin defined
the Durrmeyer type modification of these operators as
\begin{eqnarray*}
L_n^a(f;x)=\sum_{k=0}^{\infty}W_{n,k}^{a}(x)\frac{1}{B(k+1,n)}\int_{0}^{\infty}
\frac{t^k}{(1+t)^{n+k+1}}f(t)dt,  x\geq 0,
\end{eqnarray*}
and discussed some approximation properties. Very recently, Agrawal et al.
\cite{PVS} studied the simultaneous approximation and approximation of functions
having derivatives of bounded variation by these operators.\\

Later, Agrawal et al. \cite{PVA} considered a new kind of Durrmeyer type modification of the generalized
Baskakov operators having weights of Sz$\acute{a}$sz basis function
and studied some approximation properties of these operators.\\
 Motivated by the above work, we consider the Kantorovich modification of
 generalized Baskakov  operators for the function $f$ defined on
$C_{\gamma}[0,\infty):= \{f\in C[0,\infty): |f(t)| \leq M(1+t^{\gamma}), t\geq 0$ for some $\gamma>0\}$ as follows :
\begin{eqnarray}\label{e1}
K_n^a(f;x)=(n+1)\sum_{k=0}^{\infty}W_{n,k}^{a}(x)\int_{\frac{k}{n+1}}^{\frac{k+1}
{n+1}}f(t)dt,\,\,a\geq 0.
\end{eqnarray}
As a special case, for $a=0$ these operators include the well known
Baskakov-Kantorovich operators (see e.g. \cite{ZNG}).
The purpose of this paper is to study some local direct results, degree of
approximation for functions in a Lipschitz type space, approximation of
continuous functions with polynomial growth, simultaneous approximation,
statistical convergence and the approximation of absolutely continuous functions
having a derivative coinciding almost everywhere with a function of bounded
variation by the operators defined in (\ref{e1}).\\
Throughout this paper, $M$ denotes a constant not necessary the same at each occurrence.

\section{\textbf{Moment Estimates}}
For $r\in \mathbb{N}_0=\mathbb{N}\cup \{0\},$ the $r$th order moment of the
generalized Baskakov operators $B_{n,a}^*$ is defined as
\begin{eqnarray*}
\upsilon_{n,r}^a(x):= B_{n,a}^*(t^r;x) =\sum_{k=0}^{\infty}W_{n,k}^a(x)\bigg
(\frac{k}{n+1}\bigg)^{r}
\end{eqnarray*}
and the central moment of $r$th order for the operators $B_{n,a}^*$ is defined as
\begin{eqnarray*}
\mu_{n,r}^a(x) := B_{n,a}^*((t-x)^r;x) = \sum_{k=0}^{\infty}W_{n,k}^a(x)\bigg
(\frac{k}{n+1}-x\bigg)^{r}.
\end{eqnarray*}

\begin{lemma}\label{lm4}
For the function $\upsilon_{n,r}^a(x),$ we have\\
$$\upsilon_{n,0}^a(x)=1,\,\,\upsilon_{n,1}^a(x)=\frac{1}{n+1}\bigg(nx+\frac{ax}
{1+x}\bigg)$$
and
\begin{eqnarray}\label{eq3}
x(1+x)^2(\upsilon_{n,r}^a(x))^{\prime}=(n+1)(1+x)\upsilon_{n,r+1}^a(x)-(a+n(1+x))
x\,\upsilon_{n,r}^a(x).
\end{eqnarray}
Consequently,  for each $x\in[0,\infty)$ and $r\in \mathbb{N},$
\begin{eqnarray}\label{p1}
\upsilon_{n,r}^a(x)=x^r+n^{-1}(q_r(x,a)+o(1))
\end{eqnarray}
where $q_r(x,a)$ is a rational function of $x$ depending on the parameters
of $a$ and $r.$
\end{lemma}
\begin{proof}
The values of $\upsilon_{n,r}^a(x), r=0,1$ can be found by a simple computation.
Differentiating $\upsilon_{n,r}^a(x)$ with respect to $x$ and using the relation\\
$$x(1+x)^2\bigg(\frac{d}{dx}W_{n,k}^a(x)\bigg)=\bigg((k-nx)(1+x)-ax\bigg)
W_{n,k}^a(x),$$\\
we can easily get the recurrence relation (\ref{eq3}). To prove the last assertion,
we note that the equation (\ref{p1}) clearly holds for $r=1.$ The rest of the proof
follows by using (\ref{eq3}) and induction on $r.$
\end{proof}

\begin{lemma}\label{l1}
For the function $\mu_{n,r}^a(x),$ we have\\
$$\mu_{n,0}^a(x)=1,\,\, \mu_{n,1}^a(x)=\frac{1}{n+1}\bigg(-x+\frac{ax}{(1+x)}\bigg)$$
and\\
\begin{eqnarray}\label{eq1}
x(1+x)^2(\mu_{n,r}^a(x))^{\prime}=(n+1)(1+x)\mu_{n,r+1}^a(x)-ax
\mu_{n,r}^a(x)-rx(1+x)^2\mu_{n,r-1}^a(x), r\in \mathbb{N}.
\end{eqnarray}
Consequently,\\
\begin{enumerate}[(i)]
\item $\mu_{n,r}^a(x)$ is a rational function of x depending on the parameters a and r;
\\ \item for each $x\in (0,\infty)$ and $r\in \mathbb{N}_{0},\,\mu_{n,r}^a(x)
=O(n^{-[(r+1)/2]}),$ where $[\alpha]$ denotes the integer part of $\alpha.$
\end{enumerate}
\end{lemma}
\begin{proof}
Proof of this lemma follows along the lines similar to Lemma \ref{lm4}. The
consequences (i) and (ii) follow from (\ref{eq1}) by using induction on $r.$
\end{proof}

\begin{corollary}\label{c2}
If $\mu_{n,r}^{*a}(x)=\displaystyle\sum_{k=0}^{\infty}W_{n,k}^a(x)\bigg(\dfrac{k}{n}-x\bigg)^r,$
then $\mu_{n,r}^{*a}(x)=O\bigg(\dfrac{1}{n^{[\frac{r+1}{2}]}}\bigg).$
\end{corollary}
\begin{proof}
From Lemma \ref{l1}, for each $x\in[0,\infty)$ we may write
\begin{eqnarray*}
\mu_{n,r}^{*a}(x)
&=&\frac{1}{n^r}\sum_{k=0}^{\infty}W_{n,k}^a(x)(k-nx)^r\\
&=& \frac{1}{n^r}\sum_{j=0}^{r}{r\choose j}x^{r-j}(n+1)^j\bigg\{\sum_{k=0}^{\infty}
\bigg(\frac{k}{n+1}-x\bigg)^jW_{n,k}^a(x)\bigg\}\\
&=&\frac{1}{n^r} \sum_{j=0}^{r}(n+1)^jO
\bigg(\frac{1}{n^{[\frac{j+1}{2}]}}\bigg)
\\&=&\frac{1}{n^r}\sum_{j=0}^{r}O(n^{[j/2]})
= O\bigg(\frac{1}{n^{[\frac{r+1}{2}]}}\bigg).
\end{eqnarray*}
\end{proof}

\begin{lemma}\label{lm3}
For every $x\in (0,\infty)$ and $r\in \mathbb{N}_0,$ there exist
polynomials $q_{i,j,r}(x)$ in $x$ independent of $n$ and $k$ such that
\begin{eqnarray*}
\frac{d^r}{dx^r}W_{n,k}^a(x)=W_{n,k}^a(x)\displaystyle\sum_{\substack{2i+j
\leq r\\i,j\geq 0}} n^i(k-nx)^j\frac{(q_{i,j,r}(x))}{(p(x))^r},
\end{eqnarray*}
where $p(x)=x(1+x)^2.$
\end{lemma}
\begin{proof}
The proof of this lemma easily follows on proceeding along the lines of the
proof of Lemma 4 \cite{RP}. Hence the details are omitted.
\end{proof}

\begin{lemma}\label{lm2}
For the rth order $(r\in \mathbb{N}_0)$ moment of the operators (\ref{e1}),
 defined as $T_{n,r}^a(x):= K_n^a(t^r;x),$ we have
\begin{eqnarray*}
T_{n,r}^a(x)&=&\frac{1}{r+1}\sum_{j=0}^{r}{r+1\choose j}\frac{1}{(n+1)
^{r-j}}\upsilon_{n,j}^a(x),
\end{eqnarray*}
where $\upsilon_{n,j}^a(x)$ is the $j$th order moment of the operators
$B_{n,a}^*.$\\

Consequently,
$T_{n,0}^a(x)=1, \,\, T_{n,1}^a(x)= \frac{1}{n+1}\bigg(nx+\frac{ax}{1+x}
+\frac{1}{2}\bigg),$\\
$T_{n,2}^a(x)=\dfrac{1}{(n+1)^2}\bigg(n^2x^2+n\bigg(x^2+2x+\dfrac{2ax^2}
{1+x}\bigg)+\dfrac{a^2x^2}{(1+x)^2}+\dfrac{2ax}{1+x}+\dfrac{1}{3}\bigg),$\\
and for each $x\in(0,\infty)$ and $r\in \mathbb{N},  T_{n,r}^a(x)=x^r+n^{-1}
(p_r(x,a)+o(1)),$
where $p_r(x,a)$ is a rational function of $x$ depending on the parameters
$a$ and $r.$
\end{lemma}
\begin{proof}
From equation (\ref{e1}), we have
\begin{eqnarray}\label{p9}
T_{n,r}^a(x)&=&(n+1)\sum_{k=0}^{\infty}W_{n,k}^{a}(x)\int_{\frac{k}{n+1}}
^{\frac{k+1}{n+1}}t^r dt\nonumber\\
&=&\frac{n+1}{r+1}\sum_{k=0}^{\infty}W_{n,k}^a(x)\bigg\{\bigg(\frac{k+1}{n+1}\bigg)
^{r+1}-\bigg(\frac{k}{n+1}\bigg)^{r+1}\bigg\}\nonumber\\
&=& \frac{n+1}{r+1}\sum_{k=0}^{\infty}W_{n,k}^a(x)\bigg\{\sum_{j=0}^{r+1}{r+1\choose j}\bigg(\frac{k}{n+1}\bigg)^{j}\bigg(\frac{1}{n+1}\bigg)^{r+1-j}-\bigg(\frac{k}{n+1}
\bigg)^{r+1}\bigg\}\nonumber\\
&=& \frac{n+1}{r+1}\sum_{k=0}^{\infty}W_{n,k}^a(x)\bigg\{\sum_{j=0}^{r}{r+1\choose j}\bigg(\frac{k}{n+1}\bigg)^{j}\bigg(\frac{1}{n+1}\bigg)^{r+1-j}\bigg\}\nonumber\\
&=& \frac{1}{r+1}\sum_{j=0}^{r}{r+1\choose j}\frac{1}{(n+1)^{r-j}}\sum_{k=0}^{\infty}
W_{n,k}^{a}(x)\bigg(\frac{k}{n+1}\bigg)^{j}\nonumber\\
&=& \frac{1}{r+1}\sum_{j=0}^{r}{r+1\choose j}\frac{1}{(n+1)^{r-j}}\upsilon_{n,j}^a(x)
\end{eqnarray}
from which the values of $T_{n,r}^a(x), r=0,1,2$ can be found easily. The last
assertion follows from equation (\ref{p9}) by using Lemma \ref{lm4},
the required result is immediate.
\end{proof}

\begin{lemma}\label{lm1}
For the $r$th order central moment of $K_n^{a},$ defined as
\begin{eqnarray*}
u_{n,r}^a(x):= K_n^a((t-x)^r;x),
\end{eqnarray*}
we have
\begin{enumerate}[(i)]
\item  $u_{n,0}^a(x)=1, u_{n,1}^a(x)=\frac{1}{n+1}\bigg(-x+\frac{ax}{1+x}+
\frac{1}{2}\bigg)$ \\and $u_{n,2}^a(x)=\dfrac{1}{(n+1)^2}\bigg\{nx(x+1)-x(1-x)
+\dfrac{ax}{1+x}\bigg(\dfrac{ax}{1+x}+2(1-x)\bigg)+\dfrac{1}{3}\bigg\};$\\ \item
$u_{n,r}^a(x)$ is a rational function of $x$ depending on parameters a and r;\\ \item
for each $x\in (0,\infty), u_{n,r}^a(x)=O\bigg(\frac{1}{n^{[\frac{r+1}{2}]}}\bigg).$
\end{enumerate}
\end{lemma}
\begin{proof}
Using equation (\ref{e1}), assertion $(i)$ follows by a simple computation.
To prove the assertions $(ii)$ and $(iii),$ we may write
\begin{eqnarray*}
u_{n,r}^a(x)&=&(n+1)\sum_{k=0}^{\infty}W_{n,k}^{a}(x)\int_{\frac{k}{n+1}}^
{\frac{k+1}{n+1}}(t-x)^rdt\\
&=&\frac{n+1}{r+1}\sum_{k=0}^{\infty}W_{n,k}^a(x)\bigg\{\bigg(\frac{k+1}{n+1}-
x\bigg)^{r+1}-\bigg(\frac{k}{n+1}-x\bigg)^{r+1}\bigg\}\\
&=&\frac{n+1}{r+1}\sum_{k=0}^{\infty}W_{n,k}^a(x)\bigg\{\sum_{\nu=0}^{r+1}{r+1
\choose \nu}\bigg(\frac{k}{n+1}-x\bigg)^{r+1-\nu}\bigg(\frac{1}{n+1}\bigg)^{\nu}
-\bigg(\frac{k}{n+1}-x\bigg)^{r+1}\bigg\}\\&=&\frac{1}{r+1}\sum_{\nu=1}
^{r+1}{r+1\choose \nu}\frac{1}{(n+1)^{\nu-1}}\sum_{k=0}^{\infty}W_{n,k}^a(x)
\bigg(\frac{k}{n+1}-x\bigg)^{r+1-\nu}\\&=&\frac{1}{r+1}\sum_{\nu=1}^{r+1}
{r+1\choose \nu}\frac{1}{(n+1)^{\nu-1}}\mu_{n,r+1-\nu}^a(x),
\end{eqnarray*}
from which assertion $(ii)$ follows in view of Lemma \ref{l1}. Also, we get
\begin{eqnarray*}
|u_{n,r}^a(x)| &\leq& C\sum_{\nu=1}^{r+1}\frac{1}{n^{\nu-1}}\frac{1}
{n^{[\frac{r+1-\nu}{2}]}}\leq C \frac{1}{n^{[\frac{r+1}{2}]}}.
\end{eqnarray*}
This completes the proof.
\end{proof}

Let $C_B[0,\infty)$ denote the space of all real valued bounded and
continuous functions on $[0,\infty)$ endowed with the norm
\begin{eqnarray*}
\parallel f\parallel=\sup\{|f(x)|:x\in[0,\infty)\}.
\end{eqnarray*}
For $\delta>0,$ the K-functional is defined by
\begin{eqnarray*}
K_2(f,\delta)=\inf_{g\in W^2}\{\parallel f-g\parallel+\delta\parallel g''\parallel\},
 \end{eqnarray*}
where $W^{2}=\{g \in C_{B}[0,\infty): g', g^{''} \in C_{B}[0,\infty)\},$
by \cite{RAD} there exists an absolute constant $C>0$ such that
\begin{eqnarray}\label{m1}
K_{2}(f,\delta) \leq C\omega_{2}(f,\sqrt\delta),
\end{eqnarray}
where
\begin{eqnarray*}
\omega_{2}(f,\sqrt\delta)  &=& \sup_{0<h\leq\sqrt\delta} \sup_{x \in [0,\infty)}
\mid f(x+2h)-2f(x+h)+f(x) \mid
\end{eqnarray*}
is the second order modulus of smoothness of $f$. By
\begin{eqnarray*}
\omega(f,\delta) &=&  \sup_{0<h\leq\sqrt\delta} \sup_{x \in [0,\infty)}
\mid f(x+h)-f(x) \mid ,
\end{eqnarray*}
we denote the usual modulus of continuity of $f\in C_{B}[0,\infty)$.

Now, for $f\in C_B[0,\infty), x\geq 0$ the auxiliary operators are defined as
\begin{eqnarray*}
\widetilde{K}_n^a(f;x)=K_n^a(f;x)-f\bigg(\frac{1}{n+1}\bigg(nx+\frac{ax}
{1+x}+\frac{1}{2}\bigg)\bigg)+f(x).
\end{eqnarray*}

\begin{lemma}\label{l2}
Let $f\in W^2.$ Then for all $x\geq 0,$ we have
\begin{eqnarray*}
|\widetilde{K}_n^a(f;x)-f(x)|\leq \frac{1}{2}\gamma_n^a(x)\parallel f''\parallel,
\end{eqnarray*}
where
\begin{eqnarray*}
\gamma_n^a(x) &=& K_n^a((t-x)^2;x)+\frac{1}{(n+1)^2}
\bigg(-x+\frac{ax}{1+x}+\frac{1}{2}\bigg)^2\\ &=& \dfrac{1}{(n+1)^2}
\bigg\{(n+2)x^2+(n-2)x+\dfrac{2a^2x^2}{(1+x)^2}-\dfrac{4ax^2}{1+x}
+\dfrac{3ax}{1+x}+\dfrac{7}{12}\bigg\}.
\end{eqnarray*}
\end{lemma}
\begin{proof}
It is clear from the definition of $\widetilde{K}_n^a$ that\\
$$\widetilde{K}_n^a(t-x;x)=0.$$\\
Let $f\in W^2.$ From the Taylor expansion of $f$, we have
\begin{eqnarray*}
f(t)-f(x)=(t-x)f'(x)+\int_{x}^t(t-u)f''(u)du.
\end{eqnarray*}
Hence
\begin{eqnarray*}
\widetilde{K}_n^a(f;x)-f(x)&=&f'(x)\widetilde{K}_n^a(t-x;x)+\widetilde
{K}_n^a\bigg(\int_x^t(t-u)f''(u)du;x\bigg)=\widetilde{K}_n^a\bigg
(\int_x^t(t-u)f''(u)du;x\bigg)\\&=&K_n^a\bigg(\int_x^t(t-u)f''(u)du;x\bigg)
-\int_x^{\frac{1}{n+1}(nx+\frac{ax}{1+x}+\frac{1}{2})}
\bigg(\frac{1}{n+1}\bigg(nx+\frac{ax}{1+x}+\frac{1}{2}\bigg)-u \bigg)f''(u)du
\end{eqnarray*}
and thus
\begin{eqnarray}\label{e2}
|\widetilde{K}_n^a(f;x)-f(x)|\leq K_n^a\bigg(\bigg|\int_x^t(t-u)f''(u)du
\bigg|;x\bigg)+\bigg|\int_x^{\frac{1}{n+1}(nx+\frac{ax}{1+x}+\frac{1}{2})}
\bigg(\frac{1}{n+1}\bigg(nx+\frac{ax}{1+x}+\frac{1}{2}\bigg)-u \bigg
)f''(u)du\bigg|.
\end{eqnarray}
Since
\begin{eqnarray*}
\bigg|\int_x^t(t-u)f''(u)du\bigg|\leq\frac{(t-x)^2}{2}\parallel f''\parallel
\end{eqnarray*}
and
\begin{eqnarray*}
\bigg|\int_x^{\frac{1}{n+1}(nx+\frac{ax}{1+x}+\frac{1}{2})}
\bigg(\frac{1}{n+1}\bigg(nx+\frac{ax}{1+x}+\frac{1}{2}\bigg)-u \bigg)
f''(u)du\bigg|\leq\frac{1}{2(n+1)^2}\bigg(-x+\frac{ax}{1+x}
+\frac{1}{2}\bigg)^2\parallel f''\parallel,
\end{eqnarray*}
it follows from (\ref{e2}) that
\begin{eqnarray*}
|\widetilde{K}_n^a(f;x)-f(x)|&\leq&\frac{1}{2}\bigg\{K_n^a((t-x)^2;x)+\frac{1}{(n+1)^2}
\bigg(-x+\frac{ax}{1+x}+\frac{1}{2}\bigg)^2\bigg\}\parallel f''\parallel\\
&=&\frac{1}{2}\gamma_n^a(x)\parallel f''\parallel.
\end{eqnarray*}
This completes the proof of the lemma.
\end{proof}

\section{Main Results}
\subsection{Local approximation}
\begin{thm}\label{t1}
Let $f$ be a real valued bounded and uniform continuous function on $[0,\infty).$
Then for all $x\geq 0,$ there exists a constant $C>0$ such that
\begin{eqnarray*}
|K_n^a(f;x)-f(x)|\leq C\omega_2\bigg(f;\sqrt{\gamma_n^a(x)}\bigg)
+\omega\bigg(f;\frac{1}{n+1}\bigg|-x+\frac{ax}{1+x}+\frac{1}{2}\bigg|\bigg),
\end{eqnarray*}
where $\gamma_n^a(x)$ is as defined in Lemma \ref{l2}.
\end{thm}
\begin{proof}
For $f\in C_B[0,\infty)$ and $g\in W^2,$ by the definition of the
operators $\widetilde{K}_n^a,$ we obtain
\begin{eqnarray*}
|K_n^a(f;x)-f(x)|\leq |\widetilde{K}_n^a(f-g;x)|+|(f-g)(x)|+|\widetilde{K}_n^a(g;x)-g(x)|
+\bigg|f\bigg(\frac{1}{n+1}\bigg(nx+\frac{ax}{1+x}+\frac{1}{2}\bigg)\bigg)-f(x)\bigg|
\end{eqnarray*}
and
\begin{eqnarray*}
|\widetilde{K}_n^a(f;x)|\leq \parallel f\parallel K_n^a(1;x)+2\parallel
f\parallel=3\parallel f\parallel.
\end{eqnarray*}
Therefore, we have
\begin{eqnarray*}
|K_n^a(f;x)-f(x)|\leq 4\parallel f-g\parallel+|\widetilde{K}_n^a(g;x)-g(x)|
+\omega\bigg(f;\frac{1}{n+1}\bigg|-x+\frac{ax}{1+x}+\frac{1}{2}\bigg|\bigg).
\end{eqnarray*}
Now, using Lemma \ref{l2}, the above inequality reduces to
\begin{eqnarray*}
|K_n^a(f;x)-f(x)|\leq 4\parallel f-g\parallel+\gamma_n^a(x)\parallel g''\parallel
+\omega\bigg(f;\frac{1}{n+1}\bigg|-x+\frac{ax}{1+x}+\frac{1}{2}\bigg|\bigg).
\end{eqnarray*}
Thus, taking infimum over all $g\in W^2$ on the right-hand side of
the last inequality and using (\ref{m1}),\\we get the required result.
\end{proof}

Let us now consider the Lipschitz-type space in two parameters \cite{OA}:
\begin{eqnarray*}
Lip_M^{(a_1,a_2)}(\alpha):=\bigg\{f\in C_B[0,\infty):|f(t)-f(x)|\leq M\frac{|t-x|^{\alpha}}
{(t+a_1x^2+a_2x)^{\frac{\alpha}{2}}};x,t\in (0,\infty)\bigg\}\, \mbox{for}\,\, a_1, a_2>0,
\end{eqnarray*}
where $M$ is a positive constant and $\alpha\in (0,1].$

\begin{thm}\label{t2}
Let $f\in Lip_M^{(a_1,a_2)}(\alpha).$ Then, for all $x>0,$ we have
\begin{eqnarray*}
|K_n^a(f;x)-f(x)|\leq M\bigg(\frac{u_{n,2}^a(x)}{(a_1x^2+a_2x)}\bigg)^\frac{\alpha}{2}.
\end{eqnarray*}
\end{thm}
\begin{proof}
First we prove the theorem for the case $\alpha=1.$ We may write
\begin{eqnarray*}
|K_n^a(f;x)-f(x)|&\leq& (n+1)\sum_{k=0}^{\infty}W_{n,k}^{a}(x)
\int_{\frac{k}{n+1}}^{\frac{k+1}{n+1}}|f(t)-f(x)|dt\\
&\leq& M(n+1)\sum_{k=0}^{\infty}W_{n,k}^{a}(x)\int_{\frac{k}{n+1}}^
{\frac{k+1}{n+1}}\frac{|t-x|}{\sqrt{t+a_1x^2+a_2x}}dt.
\end{eqnarray*}
Using the fact that$\frac{1}{\sqrt{t+a_1x^2+a_2x}}<\frac{1}{\sqrt{a_1x^2+a_2x}}$
and the Cauchy-Schwarz inequality, the above inequality implies that
\begin{eqnarray*}
|K_n^a(f;x)-f(x)|\leq \frac{M(n+1)}{\sqrt{a_1x^2+a_2x}}\sum_{k=0}^{\infty}W_{n,k}^{a}(x)
\int_{\frac{k}{n+1}}^{\frac{k+1}{n+1}}|t-x|dt
=\frac{M}{\sqrt{a_1x^2+a_2x}}K_n^a(|t-x|;x)\leq M\bigg(\sqrt{\frac{u_{n,2}^a(x)}{a_1x^2+a_2x}}\bigg).
\end{eqnarray*}
Thus the result hold for $\alpha=1.$ Now, let $0<\alpha<1,$ then applying the H$\ddot
{o}$lder inequality with $p=\frac{1}{\alpha}$ and $q=\frac{1}{1-\alpha},$ we have
\begin{eqnarray*}
|K_n^a(f;x)-f(x)|&\leq& (n+1)\sum_{k=0}^{\infty}W_{n,k}^{a}(x)\int_{\frac{k}{n+1}}
^{\frac{k+1}{n+1}}|f(t)-f(x)|dt\\&\leq&\bigg\{\sum_{k=0}^{\infty}W_{n,k}^{a}(x)
\bigg((n+1)\int_{\frac{k}{n+1}}^{\frac{k+1}{n+1}}|f(t)-f(x)|dt\bigg)
^{\frac{1}{\alpha}}\bigg\}^{\alpha}\\&\leq& \bigg\{\sum_{k=0}^{\infty}
W_{n,k}^{a}(x)(n+1)\int_{\frac{k}{n+1}}^{\frac{k+1}{n+1}}|f(t)-f(x)|
^{\frac{1}{\alpha}}dt\bigg\}^{\alpha}\\&\leq& M\bigg\{\sum_{k=0}^{\infty}
W_{n,k}^{a}(x)(n+1)\int_{\frac{k}{n+1}}^{\frac{k+1}{n+1}}\frac{|t-x|}
{\sqrt{t+a_1x^2+a_2x}}dt\bigg\}^{\alpha}\\&\leq& \frac{M}{(a_1x^2+a_2x)^{\frac{\alpha}{2}}}
\bigg\{\sum_{k=0}^{\infty}W_{n,k}^{a}(x)(n+1)\int_{\frac{k}{n+1}}
^{\frac{k+1}{n+1}}|t-x|dt\bigg\}^{\alpha}\\&\leq& \frac{M}{(a_1x^2+a_2x)^{\frac{\alpha}{2}}}
(K_n^a(|t-x|;x))^{\alpha}\leq M\bigg(\frac{u_{n,2}^a(x)}{(a_1x^2+a_2x)}\bigg)^\frac{\alpha}{2}.
\end{eqnarray*}
Thus, the proof is completed.
\end{proof}

Next, we obtain the local direct estimate of the operators defined in (\ref{e1})
using the Lipschitz-type maximal function of order $\tau $ introduced
by B. Lenze \cite{BLN} as
\begin{equation}\label{p10}
\widetilde{\omega }_{\tau}(f,x)=\sup_{t\neq x,\,\ t\in \lbrack 0,\infty )}%
\frac{|f(t)-f(x)|}{|t-x|^{\tau }},\,\,\ x\in \lbrack 0,\infty )\,\,\ %
\mbox{and}\,\,\ \tau \in (0,1].
\end{equation}
\begin{thm}\label{t4}
 Let $f\in C_{B}[0,\infty )$ and $0<\tau \leq 1.$ Then, for all $%
x\in \lbrack 0,\infty )$ we have
\begin{equation*}
|K_n^a(f;x)-f(x)|\leq \widetilde{\omega }_{\tau
}(f,x)(u_{n,2}^a(x))^{\frac{\tau}{2}}.
\end{equation*}
\end{thm}

\begin{proof}
From the equation (\ref{p10}), we have
\begin{equation*}
|K_n^a(f;x)-f(x)|\leq \widetilde{\omega }_{\tau
}(f,x)K_n^a(|t-x|^{\tau };x).
\end{equation*}
Applying the H\"{o}lder's inequality with $p=\dfrac{2}{\tau }$ and $\dfrac{1%
}{q}=1-\dfrac{1}{p},$ we get
\begin{equation*}
|K_n^a(f;x)-f(x)|\leq \widetilde{\omega }_{\tau
}(f,x)(K_n^a(t-x)^{2};x)^{\frac{\tau }{2}}=
\widetilde{\omega }_{\tau }(f,x)(u_{n,2}^a(x))^{\frac{\tau}{2}}.
\end{equation*}
Thus, the proof is completed.
\end{proof}

\subsection{Weighted approximation}
(\cite{AAA}, \cite{VGA})  Let $H_{2}[0,\infty)$ be space of all functions $f$ defined on $[0,\infty)$
with the property that $|f(x)|\leq M_f(1+x^2),$ where $M_f$ is a constant depending
only on $f$.\\ By $C_{2}[0,\infty),$ we denote the subspace of all continuous
functions belonging to $H_{2}[0,\infty).$ If $f\in C_{2}[0,\infty)$ and
$\displaystyle\lim_{x\rightarrow \infty}|f(x)|(1+x^2)^{-1}$ exists, we write
$f\in C_{\rho}[0,\infty).$ The norm on $f\in C_{\rho}[0,\infty)$ is given by
$$\parallel f\parallel_{\rho}:=\displaystyle\sup_{x\in[0,\infty)}\frac{|f(x)|}{1+x^2}.$$
In what follows, we consider $\rho(x)=1+x^2.$\\
On the closed interval $[0,b],$ for any $b>0,$ we define the usual modulus of
continuity of $f$ by
$$\omega_b(f;\delta)=\displaystyle\sup_{|t-x|\leq \delta}\sup_{x,t\in[0,b]}|f(t)-f(x)|.$$

\begin{thm}\label{t3}
Let $f\in C_{2}[0,\infty)$ and $\omega_{b+1}(f;\delta)$ be its modulus of continuity on
the finite interval\\
$[0,b+1]\subset[0,\infty)$ with $b>0.$ Then for every $x\in [0,b]$ and $n\in \mathbb{N},$ we have
\begin{eqnarray*}
|K_n^a(f;x)-f(x)|\leq 4M_f(1+b^2)u_{n,2}^a(x)+2\omega_{b+1}
\bigg(f;\sqrt{u_{n,2}^a(x)}\bigg).
\end{eqnarray*}
\end{thm}
\begin{proof}
From \cite{GAD}, for $x\in [0,b]$ and $t\in [0,\infty),$ we have
\begin{eqnarray*}
|f(t)-f(x)|\leq \displaystyle 4 M_{f}(1+b^{2})(t-x)^{2}+\bigg(1+\frac{|t-x|}
{\delta}\bigg)\omega_{b+1}(f;\delta), \delta>0.
\end{eqnarray*}
Applying $K_n^a(.;x)$ to the above inequality and then Cauchy-Schwarz inequality
to the above inequality, we obtain
\begin{eqnarray*}
|K_n^a(f;x)-f(x)|&\leq& 4M_f(1+b^2)K_n^a((t-x)^2;x)+\omega_{b+1}(f;\delta)\bigg(1+\frac{1}
{\delta}K_n^a(|t-x|;x)\bigg)\\
&\leq& 4M_f(1+b^2)u_{n,2}^a(x)+\omega_{b+1}(f;\delta)\bigg(1+\frac{1}{\delta}\sqrt{u_{n,2}
^a(x)}\bigg).
\end{eqnarray*}
By choosing $\delta=\sqrt{u_{n,2}^a(x)},$ we obtain the desired result.
\end{proof}

\begin{thm}\label{th4}
For each $f\in C_{\rho}[0,\infty)$, we have
\begin{eqnarray*}
\lim_{n\rightarrow\infty} \parallel K_n^a(f)-f \parallel_{\rho}=0.
\end{eqnarray*}
\end{thm}
\begin{proof}
From \cite{AD}, we observe that it is sufficient to verify the
following three conditions :
\begin{eqnarray}\label{mg1}
\lim_{n\rightarrow\infty} \parallel K_n^a(t^{k};x)-x^{k}) \parallel_{\rho}=0,
\,\,  k=0,1,2.
\end{eqnarray}
Since $K_n^a(1;x)=1$, \,\, the condition in (\ref{mg1}) holds for $k=0$.
Also, by Lemma \ref{lm2} we have
\begin{eqnarray*}
\parallel K_n^a(t;x)-x)\parallel_{\rho}&=&\bigg\|\frac{1}{n+1}\bigg(-x
+\frac{ax}{1+x}+\frac{1}{2}\bigg)\bigg\|_{\rho}\\&\leq& \bigg\|\frac{1}{n+1}\bigg(a
+\frac{1}{2}-\frac{a}{1+x}-x\bigg)\bigg\|_{\rho}\\
&\leq& \frac{1}{n+1}\bigg(\bigg(a+\frac{1}{2}\bigg)\sup_{x\in [0,\infty)}
\frac{1}{1+x^2}+\sup_{x\in [0,\infty)}\frac{x}{1+x^2}+a\sup_{x\in [0,\infty)}
\frac{1}{(1+x)(1+x^2)}\bigg)\\
&\leq& \frac{1}{n+1}\bigg(2a+\frac{3}{2}\bigg),
\end{eqnarray*}
which implies that the condition in (\ref{mg1}) holds for $k=1.$\\
Similarly, we can write
\begin{eqnarray*}
\parallel K_n^a(t^{2};x)-x^{2}\parallel_{\rho} &\leq& \bigg\|\dfrac{1}{(n+1)^2}
\bigg(n^2x^2+n\bigg(x^2+2x+\dfrac{2ax^2}{1+x}\bigg)+\dfrac{a^2x^2}{(1+x)^2}
+\dfrac{2ax}{1+x}+\dfrac{1}{3}-(n+1)^2x^2\bigg)\bigg\|_{\rho}\\
&\leq& \frac{1}{(n+1)^2}\bigg((n+1)\sup_{x\in [0,\infty)}\frac{x^2}{1+x^2}
+2n\sup_{x\in [0,\infty)}\frac{x}{1+x^2}+2an\sup_{x\in [0,\infty)}\frac{x^2}
{(1+x)(1+x^2)}\\&&+2a\sup_{x\in [0,\infty)}\frac{x}{(1+x)(1+x^2)}+a^2
\sup_{x\in [0,\infty)}\frac{x^2}{(1+x)^2(1+x^2)}+\frac{1}{3}\sup_{x\in [0,\infty)}
\frac{1}{1+x^2}\bigg)\\&\leq& \frac{1}{(n+1)^2}\bigg((n+1)(2a+1)
+\bigg(2n+a^2+\frac{1}{3}\bigg)\bigg),
\end{eqnarray*}
which implies that the equation (\ref{mg1}) holds for $k=2.$\\
This completes the proof of theorem.
\end{proof}

Let $f\in C_{\rho}[0,\infty).$ The weighted modulus of continuity is defined as :
$$\Omega_{2}(f,\delta)= \displaystyle\sup_{x\geq0,0<h\leq\delta}
\frac{|f(x+h)-f(x)|}{1+(x+h)^2}.$$

\begin{lemma}\label{l6}\cite{IYN}
Let $f\in C_{\rho}[0,\infty),$ then:
\begin{enumerate}[(i)]
\item $\Omega(f,\delta)$ is a monotone increasing function of $\delta;$\\
\item $\displaystyle\lim_{\delta\rightarrow 0^{+}}\Omega(f,\delta)=0;$\\
\item for each $m\in \mathbb{N}, \Omega(f,m\delta)\leq m\Omega(f,\delta);$\\
\item for each $\lambda\in [0,\infty), \Omega(f,\lambda\delta)\leq
(1+\lambda)\Omega(f,\delta).$
\end{enumerate}
\end{lemma}

\begin{thm}\label{tm4}
Let $f\in C_{\rho}[0,\infty),$ then there exists a positive constant $M_1$ such that
\begin{eqnarray*}
\displaystyle\sup_{x\in[0,\infty)}\frac{|K_n^a(f,x)-f(x)|}{(1+x^2)^{\frac{5}{2}}}
\leq M_1\Omega\bigg(f,n^{-1/2}\bigg).
\end{eqnarray*}
\end{thm}
\begin{proof}
For $t\geq0, x\in [0,\infty)$ and $\delta>0,$ by the definition of
$\Omega(f,\delta)$ and Lemma \ref{l6}, we get
\begin{eqnarray*}
|f(t)-f(x)|&\leq& (1+(x+|x-t|^2))\Omega(f,|t-x|)\\
&\leq& 2(1+x^2)(1+(t-x)^2)\bigg(1+\frac{|t-x|}{\delta}\bigg)\Omega(f,\delta).
\end{eqnarray*}
Since $K_n^a$ is linear and positive, we have
\begin{eqnarray}\label{mg2}
|K_n^a(f,x)-f(x)|\leq2(1+x^2)\Omega(f,\delta)\bigg\{1+K_n^a((t-x)^2,x)
+K_n^a\bigg((1+(t-x)^2)\frac{|t-x|}{\delta},x\bigg)\bigg\}.
\end{eqnarray}
Using Lemma \ref{lm1}, we have
\begin{eqnarray}\label{mg3}
K_n^a((t-x)^2,x)\leq M_2\frac{1+x^2}{n}, \mbox{for some positive number} \,\,M_2.
\end{eqnarray}
Applying Cauchy-Schwarz inequality to the second term of equation
(\ref{mg2}), we have
\begin{eqnarray}\label{mg4}
K_n^a\bigg((1+(t-x)^2)\frac{|t-x|}{\delta},x\bigg)\leq\frac{1}{\delta}
\sqrt{K_n^a((t-x)^2,x)}+
\frac{1}{\delta}\sqrt{K_n^a((t-x)^4,x)}\sqrt{K_n^a((t-x)^2,x)}.
\end{eqnarray}
By Lemma \ref{lm1} and choosing $\delta=\dfrac{1}{n^{1/2}}$, there exists a positive constant $M_3$ such that
\begin{eqnarray}\label{mg5}
\sqrt{\bigg(K_n^a(t-x)^4,x)\bigg)}\leq M_3\sqrt{\frac{(1+x^2)^2}{n^2}}.
\end{eqnarray}
Combining the estimates of (\ref{mg2})-(\ref{mg5}) and taking $M_1=2(1+M_2+\sqrt{M_2}
+M_3\sqrt{M_2})$, we obtain the required result.
\end{proof}

\subsection{Simultaneous approximation}
\begin{thm}\label{th5}
Let $f\in C_{\gamma}[0,\infty).$ If $f^{(r)}$ exists at a point
$x\in(0,\infty),$ then we have
\begin{eqnarray*}
\lim_{n\rightarrow \infty}\bigg(\frac{d^r}{d\omega^r}K_n^a(f;\omega)
\bigg)_{\omega=x}=f^{(r)}(x).
\end{eqnarray*}
\end{thm}
\begin{proof}
By our hypothesis, we have
\begin{eqnarray*}
f(t)= \sum_{\nu=0}^{r}\frac{f^{(\nu)}(x)}{\nu!}(t-x)^\nu+\psi(t,x)
(t-x)^r, \,\, t\in[0,\infty),
\end{eqnarray*}
where the function $\psi(t,x)\rightarrow 0$ as $t\rightarrow x.$
From the above equation, we may write
\begin{eqnarray*}
\bigg(\frac{d^r}{d\omega^r}K_n^a(f(t);\omega)\bigg)_{\omega=x}&=&
\sum_{\nu=0}^{r}\frac{f^{(\nu)}(x)}{\nu!}\bigg(\frac{d^r}{d\omega^r}
K_n^a(t-x)^{\nu};\omega)\bigg)_{\omega=x}
+\bigg(\frac{d^r}{d\omega^r}K_n^a(\psi(t,x)(t-x)^r;\omega)\bigg
)_{\omega=x}\\ &=:& I_1+I_2, \mbox{say}.
\end{eqnarray*}
Now, we estimate $I_1.$
\begin{eqnarray*}
I_1&=&\sum_{\nu=0}^{r}\frac{f^{(\nu)}(x)}{\nu!}\bigg\{\frac{d^r}{d\omega^r}
\bigg(\sum_{j=0}^{\nu}{\nu \choose j}(-x)^{\nu-j}K_{n}^a(t^j;\omega)\bigg)
_{\omega=x}\bigg\}\\&=& \sum_{\nu=0}^{r}\frac{f^{(\nu)}(x)}{\nu!}\sum_{j=0}
^{\nu}{\nu \choose j}(-x)^{\nu-j}\bigg(\frac{d^r}{d\omega^r}K_{n}^a(t^j;\omega)
\bigg)_{\omega=x}\\&=& \sum_{\nu=0}^{r-1}\frac{f^{(\nu)}(x)}{\nu!}\sum_{j=0}
^{\nu}{\nu \choose j}(-x)^{\nu-j}\bigg(\frac{d^r}{d\omega^r}K_{n}^a(t^j;\omega)
\bigg)_{\omega=x}+\frac{f^{(r)}(x)}{r!}\sum_{j=0}^{r}{r \choose j}(-x)^{r-j}
\bigg(\frac{d^r}{d\omega^r}K_{n}^a(t^j;\omega)\bigg)_{\omega=x}\\
&:=& I_3+I_4, \mbox{say.}
\end{eqnarray*}
First, we estimate $I_4.$
\begin{eqnarray*}
I_4&=&\frac{f^{(r)}(x)}{r!}\sum_{j=0}^{r-1}{r \choose j}(-x)^{r-j}
\bigg(\frac{d^r}{d\omega^r}K_{n}^a(t^j;\omega)\bigg)_{\omega=x}
+\frac{f^{(r)}(x)}{r!}\bigg(\frac{d^r}{d\omega^r}K_{n}^a(t^r;\omega)\bigg)
_{\omega=x}\\&:=& I_5+I_6, \mbox{say}.
\end{eqnarray*}
By using Lemma \ref{lm2}, we get\\
$I_6= f^{(r)}(x)+O\bigg(\dfrac{1}{n}\bigg), I_3=O\bigg(\dfrac{1}{n}\bigg)$
and $I_5= O\bigg(\dfrac{1}{n}\bigg).$\\
Combining the above estimates, for each $x\in (0,\infty)$ we obtain $I_1
\rightarrow f^{(r)}(x)$ as $n\rightarrow \infty.$\\

Since $\psi(t,x)\rightarrow 0$ as $t\rightarrow x,$ for a given $\epsilon>0$
there exists a $\delta>0$ such that $|\psi(t,x)|<\epsilon$ whenever $|t-x|<\delta.$
For $|t-x|\geq\delta, |\psi(t,x)|\leq M|t-x|^{\beta},$ for some $M, \beta>0.$\\
By making use of Lemma \ref{lm3}, we have
\begin{eqnarray*}
|I_2|\leq (n+1)\sum_{k=0}^{\infty}\sum_{\substack {2i+j\leq r\\ i,j\geq 0}}
n^i|k-nx|^j\frac{|q_{i,j,r}(x)|}{(p(x))^r}W_{n,k}^a(x)\int_{\frac{k}{n+1}}
^{\frac{k+1}{n+1}}\psi(t,x)(t-x)^r dt.
\end{eqnarray*}
\begin{eqnarray*}
|I_2|&\leq& (n+1)\sum_{k=0}^{\infty}\sum_{\substack {2i+j\leq r\\ i,j\geq 0}} n^i |k-nx|^j\frac{|q_{i,j,r}(x)|}{(p(x))^r}W_{n,k}^a(x)\bigg(\epsilon\int_{|t-x|<\delta}
|t-x|^r dt+M\int_{|t-x|\geq \delta}|t-x|^{r+\beta}dt\bigg)\\
&:=&I_7+I_8, \,\,\mbox{say}.
\end{eqnarray*}
Let $S=\displaystyle\sup_{\substack {2i+j\leq r\\ i,j\geq 0}}\frac{|q_{i,j,r}(x)|}
{(p(x))^r}$ and by applying the Schwarz inequality, Corollary \ref{c2} and Lemma \ref{lm1},
we get
\begin{eqnarray*}
|I_7|&\leq& \epsilon(n+1)^{\frac{1}{2}}S\sum_{k=0}^{\infty}\sum_{\substack {2i+j\leq r\\
i,j\geq 0}} n^i|k-nx|^jW_{n,k}^a(x)\bigg(\int_{\frac{k}{n+1}}
^{\frac{k+1}{n+1}}(t-x)^{2r} dt\bigg)^{\frac{1}{2}}\\&\leq& \epsilon(n+1)
^{\frac{1}{2}}S\sum_{\substack {2i+j\leq r\\ i,j\geq 0}}n^i \bigg(\sum_{k=0}
^{\infty}(k-nx)^{2j}W_{n,k}^a(x)\bigg)^{\frac{1}{2}}\bigg(\sum_{k=0}
^{\infty}W_{n,k}^a(x)\int_{\frac{k}{n+1}}^{\frac{k+1}{n+1}}(t-x)^{2r} dt\bigg)
^{\frac{1}{2}}\\&=& \epsilon . S \sum_{\substack {2i+j\leq r\\ i,j\geq 0}}
n^i \bigg(\sum_{k=0}^{\infty}(k-nx)^{2j}W_{n,k}^a(x)\bigg)^{\frac{1}{2}}
\bigg((n+1)\sum_{k=0}^{\infty}W_{n,k}^a(x)\int_{\frac{k}{n+1}}^{\frac{k+1}{n+1}}
(t-x)^{2r} dt\bigg)^{\frac{1}{2}}\\&\leq& \epsilon . S \sum_{\substack
{2i+j\leq r\\ i,j\geq 0}}O\bigg(n^{\frac{2i+j}{2}}\bigg)O\bigg(n^{-r/2}\bigg)
=\epsilon . O(1).
\end{eqnarray*}
Since $\epsilon >0$ is arbitrary, $I_7\rightarrow 0$ as $n\rightarrow \infty.$
Let $s(\in \mathbb{N})>r+\beta.$ Again, by using Schwarz inequality, Corollary \ref{c2}
and Lemma \ref{lm1}, we obtain
\begin{eqnarray*}
I_8&\leq& M S(n+1)\sum_{k=0}^{\infty}\sum_{\substack {2i+j\leq r\\ i,j\geq 0}}
n^i|k-nx|^jW_{n,k}^a (x)\int_{|t-x|\geq \delta}|t-x|^{r+\beta}dt\\
&\leq& \frac{M^{\prime}(n+1)}{\delta^{s-r-\beta}}\sum_{k=0}^{\infty}\sum_{\substack
{2i+j\leq r\\ i,j\geq 0}}n^i|k-nx|^jW_{n,k}^a (x)\int_{\frac{k}{n+1}}^{\frac
{k+1}{n+1}}|t-x|^{s} dt,  \,\, \mbox{where}\,\, M S=M^{\prime}\\
&\leq&  \frac{M^{\prime}(n+1)^{1/2}}{\delta^{s-r-\beta}}\sum_{k=0}
^{\infty}\sum_{\substack{2i+j\leq r\\ i,j\geq 0}}n^i|k-nx|^jW_{n,k}^a(x)
\bigg(\int_{\frac{k}{n+1}}^{\frac{k+1}{n+1}}|t-x|^{2s} dt\bigg)^{1/2}\\
&\leq&  \frac{M^{\prime}}{\delta^{s-r-\beta}}\sum_{\substack {2i+j\leq r\\
i,j\geq 0}}n^i\bigg(\sum_{k=0}^{\infty}W_{n,k}^a(x)(k-nx)^{2j}\bigg)^{1/2}
\bigg((n+1)\sum_{k=0}^{\infty}W_{n,k}^a(x)\int_{\frac{k}{n+1}}^{\frac{k+1}{n+1}}
(t-x)^{2s} dt\bigg)^{1/2}\\&=&  \frac{M^{\prime}}{\delta^{s-r-\beta}}
\sum_{\substack {2i+j\leq r\\ i,j\geq 0}}n^iO(n^{j/2})O(n^{-s/2})
= \frac{M^{\prime}}{\delta^{s-r-\beta}}O(n^{\frac{(r-s)}{2}})
\end{eqnarray*}
which implies that $I_8\rightarrow 0,$ as $n\rightarrow \infty.$\\
Now, by combining the estimates of $I_7$ and $I_8,$ we get $I_2\rightarrow 0$
as $n\rightarrow \infty.$\\
Thus, from the estimates of $I_1$ and $I_2,$ we obtained the required result.
\end{proof}

\subsection{Voronovskaja type result}
\begin{thm}\label{thm6}
Let $f\in C_{\gamma}[0,\infty).$  If $f$ admits a derivative of order $(r+2)$ at a
fixed point $x\in (0,\infty),$ then we have
\begin{eqnarray*}
\lim_{n\rightarrow\infty}n\bigg(\bigg(\frac{d^r}{d\omega^r}K_n^a(f;\omega)\bigg)
_{\omega=x}-f^{(r)}(x)\bigg)=\displaystyle\sum_{\nu=1}^{r+2}Q(\nu,r,a,x)f^{(\nu)}(x),
\end{eqnarray*}
where $Q(\nu,r,a,x)$ are certain rational functions of $x$ depending on the
parameter $a,r,\nu.$
\end{thm}
\begin{proof}
From the Taylor's theorem, we have
\begin{eqnarray}\label{eq2}
f(t)=\sum_{\nu=0}^{r+2}\frac{f^{(\nu)}(x)}{\nu!}(t-x)^{\nu}+\psi(t,x)(t-x)^{r+2},
 \,\, t\in[0,\infty)
\end{eqnarray}
where $\psi(t,x)\rightarrow 0$ as $t\rightarrow x$ and $\psi(t,x)=O(t-x)^{\gamma}.$\\
From the equation (\ref{eq2}), we have
\begin{eqnarray*}
\bigg(\frac{d^r}{d\omega^r}K_n^a(f(t);\omega)\bigg)_{\omega=x}&=&\sum_{\nu=0}^{r+2}
\frac{f^{(\nu)}(x)}{\nu!}\bigg(\frac{d^r}{dw^{r}}(K_n^a((t-x)^{\nu};\omega)\bigg)
_{\omega=x}+\bigg(\frac{d^r}{d\omega^r}K_n^a(\psi(t,x)(t-x)^{r+2};\omega)
\bigg)_{\omega=x}\\&=& \sum_{\nu=0}^{r+2}\frac{f^{(\nu)}(x)}{\nu!}
\sum_{j=0}^{\nu}{\nu\choose j}(-x)^{\nu-j}\bigg(\frac{d^r}{d\omega^r}K_n^a(t^j;\omega)
\bigg)_{\omega=x}+\bigg(\frac{d^r}{d\omega^r}K_n^a(\psi(t,x))(t-x)^{r+2};\omega\bigg)
_{\omega=x}\\&:=& I_1+I_2, \,\,  \mbox{say}.
\end{eqnarray*}
Proceeding in a manner to the estimates of $I_2$ in Theorem \ref{th5},
for each $x\in (0,\infty)$ we get
\begin{eqnarray*}
\displaystyle\lim_{n\rightarrow \infty}n\bigg(\frac{d^r}{d\omega^r}(K_n^a
(\psi(t,x)(t-x)^{r+2};\omega)\bigg)_{\omega=x}=0.
\end{eqnarray*}
Now, we estimate $I_1.$
\begin{eqnarray*}
I_1 &=&\sum_{\nu=0}^{r-1}\frac{f^{(\nu)}(x)}{\nu!}\sum_{j=0}^{\nu}{\nu\choose j}
(-x)^{\nu-j}\bigg(\frac{d^r}{d\omega^r}K_n^a(t^j;\omega)\bigg)_{\omega=x}
+\frac{f^{(r)}(x)}{r!}\sum_{j=0}^{r}{r\choose j}(-x)^{r-j}\bigg(\frac{d^r}
{d\omega^r}K_n^a(t^j;\omega)\bigg)_{\omega=x}\\
&&+\frac{f^{(r+1)}(x)}{(r+1)!}\sum_{j=0}^{r+1}{r+1\choose j}(-x)^{r+1-j}\bigg
(\frac{d^r}{d\omega^r}K_n^a(t^j;\omega)\bigg)_{\omega=x}\\
&&+\frac{f^{(r+2)}(x)}{(r+2)!}\sum_{j=0}^{r+2}{r+2\choose j}(-x)^{r+2-j}\bigg
(\frac{d^r}{d\omega^r}K_n^a(t^j;\omega)\bigg)_{\omega=x}.
\end{eqnarray*}
By making use of Lemma \ref{lm2}, we have\\
$I_1=f^{(r)}(x)+n^{-1}\bigg(\displaystyle\sum_{\nu=1}^{r+2}Q(\nu,r,a,x)f^{(\nu)}
(x)+o(1)\bigg).$\\
Thus, from the estimates of $I_1$ and $I_2,$ the required result follows. This
completes the proof.

\begin{corollary}\label{c1}
From the above theorem, we have
\begin{enumerate}[(i)]
\item for $r=0\\
\displaystyle\lim_{n\rightarrow \infty}n(K_n^a(f;x)-f(x))= \bigg(\frac{ax}{1+x}
+\frac{1}{2}-x\bigg)f'(x)+\frac{1}{2}(x+x^2)f''(x);$\\
\item for $r=1\\
\displaystyle\lim_{n\rightarrow
\infty}n\bigg(\bigg(\frac{d}{d\omega}K_n^a(f;\omega)-f^{\prime}(x)\bigg)_{\omega=x}\bigg)
=\bigg(-1+\frac{a}{(1+x)^2}\bigg)f'(x)+\bigg(1+\frac{ax}{1+x}\bigg)f''(x)
+\frac{1}{2}x(1+x)f'''(x).$
\end{enumerate}
\end{corollary}
\end{proof}

\subsection{Degree of approximation}
In this section, we obtain an estimate of the degree of approximation for
$r$th order derivative of $K_n^a$ for smooth functions.
\begin{thm}
Let $r\leq q \leq r+2, f\in C_{\gamma}[0,\infty)$ and $f^{(q)}$ exist and
be continuous on $(a-\eta, b+\eta), \eta>0$. Then, for sufficiently large $n$
\begin{eqnarray*}
\bigg\|\bigg(\frac{d^r}{d\omega^r} K_n^a(f;\omega)\bigg)_{\omega=t}-f^{(r)}(t)
\bigg\|_{C[a,b]}\leq  C_1 n^{-(q-r)/2}\omega(f^{(q)},n^{-1/2})+ C_2\,\,n^{-1},
\end{eqnarray*}
where $C_1=C_1(r)$ and $C_2=C_2(r,f).$
\end{thm}
\begin{proof}
By our hypothesis we have
\begin{eqnarray}\label{p2}
f(t)=\sum_{i=0}^{q}\frac{f^{(i)}(x)}{i!}(t-x)^i +\frac{f^{(q)}(\xi)-f^{(q)}(x)}
{q!}(t-x)^{q} \chi(t)+\phi(t,x)(1-\chi(t)),
\end{eqnarray}
where $\xi$ lies between $t$ and $x$ and $\chi(t)$ is the characterstic function
of $(a-\eta, b+\eta).$
The function $\phi(t,x)$ for $t\in[0,\infty)\setminus (a-\eta, b+\eta)$ and $x\in [a,b]$ is bounded by $M|t-x|^{\kappa}$ for some
constants $M, \kappa >0.$\\
Operating $\dfrac{d^r}{d\omega^r} K_n^a(.;\omega)$ on the equality (\ref{p2})
and breaking the right hand side into three parts $J_1, J_2$ and $J_3,$ say,
corresponding to the three terms on the right hand side of equation (\ref{p2})
as in the estimate of $I_8$ in Theorem \ref{th5},
it can be easily shown that $J_3=o(n^{-1}),$ uniformly in $x\in
[a,b].$\\ Now treating $J_1$ in a manner similar to the treatment of $I_1$ of Theorem
\ref{thm6}, we get $J_1=f^{(r)}(t)+O(n^{-1}),$ uniformly in $t\in [a,b].$\\
Finally, let
\begin{eqnarray*}
S_1=\displaystyle\sup_{x\in[a,b]}\sup_{\substack {2i+j\leq r\\ i,j\geq 0}}
\frac{q_{i,j,r}(x)}{(p(x))^r},
\end{eqnarray*}
then making use of the inequality
\begin{eqnarray*}
|f^{(q)}(\xi)-f^{(q)}(x)|\leq \bigg(1+\frac{|t-x|}{\delta}\bigg)\omega (f^{(q)},
\delta),\, \delta>0,
\end{eqnarray*}
the Schwarz inequality, Corollary \ref{c2} and Lemma \ref{lm1}, we obtain
\begin{eqnarray*}
|J_2|&\leq& (n+1)\sum_{k=0}^{\infty}\sum_{\substack {2i+j\leq r\\ i,j\geq 0}}
\frac{n^i|k-nx|^j q_{i,j,r}(x)}{(p(x))^r}W_{n,k}^a(x)\int_{\frac{k}{n+1}}
^{\frac{k+1}{n+1}}\frac{|f^{(q)}(\xi)-f^{(q)}(x)|}{q!}|t-x|^q \chi(t) dt\\
&\leq& \frac{\omega (f^{(q)},\delta) S_1}{q!} \sum_{\substack {2i+j\leq r\\ i,j\geq 0}}
 n^i\bigg(\sum_{k=0}^{\infty}(k-nx)^{2j}W_{n,k}^a(x)\bigg)^{1/2}\bigg\{\bigg
((n+1)\sum_{k=0}^{\infty}W_{n,k}^a(x)\int_{\frac{k}{n+1}}^{\frac{k+1}{n+1}}
(t-x)^{2q} dt\bigg)^{1/2}\\&&+\frac{1}{\delta}\bigg((n+1)\sum_{k=0}^{\infty}
W_{n,k}^a(x)\int_{\frac{k}{n+1}}^{\frac{k+1}{n+1}}(t-x)^{2q+2} dt\bigg)^{1/2}
\bigg\}\\&\leq& C_1\bigg(n^{-(q-r)/2}\bigg)\omega(f^{(q)},n^{-1/2}),\,\,
 \mbox{on choosing}\,\,\delta=n^{-1/2}.
\end{eqnarray*}
By combining the estimates of $J_1-J_3,$ we get the required result.
\end{proof}

\subsection{Statistical convergence}
Let $A=(a_{nk})$ be a non-negative infinite summability matrix. For a given
sequence $x:=(x)_{n},$ the A-transform of x denoted by $Ax:(Ax)_{n}$ is
defined as\newline
\begin{equation*}
(Ax)_{n}=\displaystyle\sum_{k=1}^{\infty }a_{nk}x_{k}
\end{equation*}%
provided the series converges for each n. A is said to be regular if $%
\displaystyle\lim_{n}(Ax)_{n}=L$ whenever $\displaystyle\lim_{n}(x)_{n}=L.$
Then $x=(x)_{n}$ is said to be A-statistically convergent to L i.e. $st_{A}-%
\displaystyle\lim_{n}(x)_{n}=L$ if for every\newline
$\epsilon >0,\,\,\displaystyle\lim_{n}\sum_{k:|x_{k}-L|\geq \epsilon
}a_{nk}=0.$ If we replace $A$ by $C_{1}$ then $A$ is a Cesaro matrix of
order one and $A$-statistical convergence is reduced to the statistical
convergence. Similarly, if $A=I,$ the identity matrix then $A$-statistical
convergence is called ordinary convergence.

\begin{thm}
Let $(a_{nk})$ be a non-negative regular summability matrix and
$x\in[0,\infty).$ Then, for all  $f\in C_{\rho}[0,\infty)$ we have
\begin{eqnarray*}
st_{A}-\lim_{n}\|K_{n}^a(f,.)-f\|_{\rho_{\alpha}}=0,
\end{eqnarray*}
where $\rho_{\alpha}(x)=1+x^{2+\alpha},$ $\alpha>0.$
\end{thm}

\begin{proof}
From (\cite{ODC}, p. 191, Th. 3), it is sufficient to show that
$st_{A}-\lim_{n}\|K_{n}^a(e_i,.)-e_i\|_{\rho}=0,$
where $e_i(x)=x^{i},$ $i=0,1,2.$\\

In view of Lemma \ref{lm2}, it follows that
\begin{eqnarray}
st_{A}-\lim_{n}\|K_{n}^a(e_0,.)-e_0\|_{\rho}=0.
\end{eqnarray}
Again, by using Lemma \ref{lm2}, we have
\begin{eqnarray*}
\sup_{x\in[0,\infty)}\frac{|K_{n}^a(e_1,x)-e_1(x)|}{1+x^2}&=&\displaystyle
\sup_{x\in[0,\infty)}\frac{\bigg|\dfrac{1}{n+1}\bigg(-x+\dfrac{ax}{1+x}
+\dfrac{1}{2}\bigg)\bigg|}{1+x^2}\\&\leq& \frac{1}{n+1}\bigg(a+\frac{3}{2}\bigg).
\end{eqnarray*}
For $\epsilon>0,$ we define the following sets
\begin{eqnarray*}
D:&=&\bigg\{n:\|K_{n}^a(e_1,.)-e_1\|_{\rho}\geq \epsilon\bigg\}\\
D_1:&=&\bigg\{n:\frac{1}{n+1}\bigg(a+\frac{3}{2}\bigg)\geq\epsilon\bigg\}.
\end{eqnarray*}
which yields us $D\subseteq D_{1}$ and therefore for all $n,$ we have
$\displaystyle\sum_{k\in D}a_{nk}\leq \sum_{k\in D_1}a_{nk}$ and hence
\begin{eqnarray}
st_{A}-\lim_{n}\|K_{n}^a(e_1,.)-e_1\|_{\rho}=0.
\end{eqnarray}
Proceeding similarly,
\begin{eqnarray*}
\parallel K_{n}^a(e_{2};.)-e_{2}\parallel _{\rho} &=&
\sup_{x\in[0,\infty)}\frac{1}{1+x^2}\bigg| \frac{n}{(n+1)^2}
\bigg(-x^2+2x+\frac{2ax^2}{1+x}\bigg)+\frac{1}{(n+1)^2}\bigg(\frac{a^2x^2}
{(1+x)^2}+\frac{2ax}{1+x}-x^2+\frac{1}{3}\bigg)\bigg|\\
&\leq& \frac{1}{n+1}(2a+3)+\frac{1}{(n+1)^2}\bigg(a^2+4a+\frac{13}{3}\bigg).
\end{eqnarray*}
Let us define the following sets
\begin{eqnarray*}
E:&=&\bigg\{n:\|K_{n}^a(e_2,.)-e_2\|_{\rho}\geq \epsilon\bigg\}\\
E_1:&=&\bigg\{n:\frac{1}{n+1}(2a+3)\geq\epsilon/2\bigg\}\\
E_2:&=&\bigg\{n:\frac{1}{(n+1)^2}\bigg(a^2+4a+\frac{13}{3}
\bigg)\geq\epsilon/2\bigg\}.
\end{eqnarray*}
Then, we obtain $E\subseteq E_{1}\cup E_{2},$ which implies that
\begin{equation*}
\displaystyle\sum_{k\in E}a_{nk}\leq \displaystyle\sum_{k\in E_{1}}a_{nk}
+\displaystyle\sum_{k\in E_{2}}a_{nk}
\end{equation*}
and hence
\begin{eqnarray}
st_{A}-\lim_{n}\|K_{n}^a(e_2,.)-e_2\|_{\rho}=0.
\end{eqnarray}
This completes the proof of the theorem.
\end{proof}

\subsection{Rate of approximation}
The rate of convergence for functions with derivative of bounded variation is an
interesting area of research in approximation theory. A pioneering work in this
direction is due to Bojanic and Cheng (\cite{RB1}, \cite{RB2})
who estimated the rate of convergence with derivatives of bounded variation for
Bernstein and Hermite-Fejer polynomials by using different methods. After that many
researchers have obtained results in this direction for different sequences of
linear positive operators.\\
Now, we shall estimate the rate of convergence for the generalized Baskakov
Kantorovich Operators $K_n^a$ for functions with derivatives of bounded variation
defined on $(0,\infty)$ at points $x$ where $f^{\prime}(x+)$ and $f^{\prime}(x-)$
exist, we shall prove that the operators (\ref{e1}) converge to the limit $f(x).$\\

Let $f\in DBV_{\gamma}(0,\infty),$ $\gamma\geq 0$ be the class of all functions
defined on $(0,\infty),$ having a derivative of bounded variation on every finite
subinterval of $(0,\infty)$ and $|f(t)|\leq M t^{\gamma},$ $\forall\,\,\ t>0.$\\
We notice that the functions $f\in DBV_{\gamma}(0,\infty)$ possess a representation
\begin{eqnarray*}
f(x)=\int_{0}^{x}h(t)dt+f(0),
\end{eqnarray*}
where $h(t)$ is a function of bounded variation on each finite subinterval of
$(0,\infty).$

The operators $K_n^a(f;x)$ also admit the integral representation
\begin{eqnarray}\label{pa}
K_n^a(f;x)=\int_{0}^{\infty}J_n^a(x,t) f(t)dt,
\end{eqnarray}
where $J_n^a(x,t):= (n+1)\displaystyle\sum_{k=0}^{\infty}W_{n,k}^a(x)\chi_{n,k}(t),$
where $\chi_{n,k}(t)$ is the characteristic function of the interval
$\bigg[\frac{k}{n+1}, \frac{k+1}{n+1}\bigg]$ with respect to $[0,\infty).$

\begin{remark}\label{r1}
From Lemma \ref{lm1}, for $\lambda>1, x\in(0,\infty)$ and $n$ sufficiently
large, we have
\begin{eqnarray*}
K_n^a((t-x)^2;x)=u_{n,2}^a(x)\leq \frac{\lambda x(1+x)}{n+1}.
\end{eqnarray*}
\end{remark}

In order to prove the main result, we need the following Lemma.
\begin{lemma}\label{lm5}
For fixed $x\in (0,\infty), \lambda >1$ and $n$ sufficiently large, we have
\begin{enumerate}[(i)]
\item $\alpha_n^a(x,y)=\int_{0}^y J_n^a(x,t)dt\leq \dfrac{1}{(x-y)^2}\dfrac
{\lambda x(1+x)}{n+1},\,\, 0\leq y < x,$\\
\item $1-\alpha_n^a(x,z)= \int_{z}^{\infty} J_n^a(x,t)dt\leq\dfrac{1}{(z-x)^2}
\dfrac{\lambda x(1+x)}{n+1},\,\, x < z < \infty.$
\end{enumerate}
\end{lemma}
\begin{proof}
First we prove $(i).$
\begin{eqnarray*}
\alpha_n^a(x,y)&=&\int_{0}^y J_n^a(x,t)dt\leq \int_{0}^y\bigg(\frac{x-t}
{x-y}\bigg)^2 J_n^a(x,t)dt\\&\leq& \frac{1}{(x-y)^2}K_n^a((t-x)^2;x)\\
&\leq&\dfrac{1}{(x-y)^2}\dfrac{\lambda x(1+x)}{n+1}.
\end{eqnarray*}
The proof of $(ii)$ is similar.
\end{proof}

\begin{thm}
Let $f\in DBV_{\gamma}(0,\infty).$ Then, for every $x\in (0,\infty),$ and $n$
sufficiently large, we have
\begin{eqnarray*}
|K_n^a(f;x)-f(x)| &\leq& \dfrac{\bigg|-x+\dfrac{ax}{1+x}+\dfrac{1}{2}
\bigg|}{n+1}\dfrac{|f^{\prime}(x+)+f^{\prime}(x-)|}{2}+\sqrt{\frac{\lambda x(1+x)}
{n+1}}\dfrac{|f^{\prime}(x+)-f^{\prime}(x-)|}{2}\\&&+\dfrac{\lambda (1+x)}
{n+1}\sum_{k=1}^{[\sqrt{n}]}\bigvee_{x-(x/k)}^{x}(f_x^{\prime})+\frac{x}{\sqrt{n}}
\bigvee_{x-(x/\sqrt{n})}^{x}(f_x^{\prime})\\&&+\dfrac{\lambda (1+x)}{n+1}
\sum_{k=0}^{[\sqrt{n}]}\bigvee_{x}^{x+x/k}(f_x^{\prime})+\frac{x}{\sqrt{n}}
\bigvee_x^{x+x/\sqrt{n}}(f_x^{\prime}),
\end{eqnarray*}
where $\bigvee_c^d f(x)$ denotes the total variation of $f(x)$ on $[c,d]$ and $f_x$
is defined by
\begin{eqnarray*}
f_x(t)=\left\{
\begin{array}{cc}
f(t)-f(x-), &0\leq t<x\\
0 \,\,\,\,, & t=x\\
f(t)-f(x+), & x<t<\infty.
\end{array}
\right.
\end{eqnarray*}
\end{thm}
\begin{proof}
For $u\in [0,\infty),$ we may write
\begin{eqnarray}\label{p3}
f^{\prime}(u)&=& f_x^{\prime}(u)+\frac{1}{2}(f^{\prime}(x+)+f^{\prime}(x-))
+\frac{1}{2}(f^{\prime}(x+)-f^{\prime}(x-))sgn(u-x)\nonumber\\&&+\delta_x(u)
\{f^{\prime}(u)-\frac{1}{2}(f^{\prime}(x+)+f^{\prime}(x-))\},
\end{eqnarray}
where
\begin{eqnarray*}
\delta_x(u)=\left\{
\begin{array}{cc}
1 \,\,\,\,, & u=x\\
0 \,\,\,\,, & u\neq x.
\end{array}
\right.
\end{eqnarray*}
From (\ref{pa}) we get
\begin{eqnarray}\label{p4}
K_n^a(f;x)-f(x)&=&\int_0^{\infty}J_n^a(x,t)(f(t)-f(x))dt\nonumber\\
&=&\int_0^{\infty}J_n^a(x,t)\int_x^{t}f^{\prime}(u)du dt.
\end{eqnarray}
It is obvious that
\begin{eqnarray*}
\int_0^{\infty}\bigg(\int_{x}^t\bigg(f^{\prime}(u)-\frac{1}{2}(f^{\prime}(x+)
+f^{\prime}(x-))\bigg)\delta_x(u)du\bigg)J_n^a(x,t)dt=0.
\end{eqnarray*}
By (\ref{p4}), we have
\begin{eqnarray*}
\int_0^{\infty}\bigg(\int_{x}^t\frac{1}{2}(f^{\prime}(x+)+f^{\prime}(x-))du
\bigg)J_n^a(x,t)dt&=&\frac{1}{2}(f^{\prime}(x+)+f^{\prime}(x-))\int_0^{\infty}
(t-x)J_n^a(x,t)dt\\=&&\frac{1}{2}(f^{\prime}(x+)+f^{\prime}(x-))K_n^a((t-x);x)
\end{eqnarray*}
and by using Schwarz inequality, we obtain
\begin{eqnarray*}
\bigg|\int_0^{\infty}J_n^a(x,t)\bigg(\int_{x}^t\frac{1}{2}(f^{\prime}(x+)-f^{\prime}(x-))
sgn(u-x)du\bigg)dt\bigg|&\leq& \frac{1}{2} |f^{\prime}(x+)-f^{\prime}(x-)|\int_0^{\infty}
|t-x|J_n^a(x,t)dt\\&\leq&\frac{1}{2} |f^{\prime}(x+)-f^{\prime}(x-) |K_n^a(|t-x|;x)\\
&\leq& \frac{1}{2} |f^{\prime}(x+)-f^{\prime}(x-)| (K_n^a((t-x)^2;x))^{1/2}.
\end{eqnarray*}
From Lemma \ref{lm2}, Remark \ref{r1} and from the above estimates (\ref{p4}) becomes
\begin{eqnarray}\label{p5}
|K_n^a(f;x)-f(x)|&\leq& \frac{1}{2(n+1)}|f^{\prime}(x+)+f^{\prime}(x-)|\bigg|-x
+\frac{ax}{1+x}+\frac{1}{2}\bigg|\nonumber\\&&+\frac{1}{2}|f^{\prime}(x+)-f^{\prime}(x-)|
\sqrt{\frac{\lambda x(1+x)}{n+1}}+\bigg|\int_0^x \bigg(\int_x^t f_x^{\prime}(u)du\bigg)
J_n^a(x,t)dt\nonumber\\&&+\int_x^{\infty}\bigg(\int_x^t f_x^{\prime}(u)du\bigg)J_n^a(x,t)dt\bigg|.
\end{eqnarray}
Let\\
$U_n(f^{\prime},x)=\displaystyle\int_0^x \bigg(\int_x^t f_x^{\prime}(u)du\bigg)
J_n^a(x,t)dt,$\\
and\\
$V_n(f^{\prime},x)= \displaystyle\int_x^{\infty}\bigg(\int_x^t f_x^{\prime}(u)du
\bigg)J_n^a(x,t)dt.$\\
Now, we estimate the terms $U_n(f^{\prime},x)$ and $V_n(f^{\prime},x).$
Since $\int_{c}^d d_t \alpha_n^a(x,t)\leq 1$ for all $[c,d]\subseteq (0,\infty),$\\
using integration by parts and Lemma \ref{lm5} with $y=x-\frac{x}{\sqrt{n}}$
we have
\begin{eqnarray*}
|U_n(f^{\prime},x)|&=&\bigg|\int_0^x \bigg(\int_x^t f_x^{\prime}(u)du\bigg)d_t
\alpha_n^a(x,t)dt\bigg|\\&=&\bigg|\int_0^x\alpha_n^a(x,t)f_x^{\prime}(t)dt\bigg|
\\&\leq& \bigg(\int_0^y+\int_y^x\bigg)|f_x^{\prime}(t)||\alpha_n^a(x,t)|dt\\
&\leq&\frac{\lambda x(1+x)}{n+1}\int_0^y \bigvee_t^x(f_x^{\prime})(x-t)^{-2}dt
+\int_y^x \bigvee_t^x(f_x^{\prime})dt\\ &\leq& \frac{\lambda x(1+x)}{n+1}
\int_0^y \bigvee_t^x(f_x^{\prime})(x-t)^{-2}dt+\frac{x}{\sqrt{n}}
\bigvee_{x-(x/\sqrt{n})}^x(f_x^{\prime}).
\end{eqnarray*}
By the substitution of $u=\frac{x}{x-t},$ we get
\begin{eqnarray*}
\frac{\lambda x(1+x)}{n+1}\int_0^{x-x/\sqrt{n}}(x-t)^{-2}\bigvee_t^x
(f_x^{\prime})dt&=&\frac{\lambda(1+x)}{n+1}\int_1^{\sqrt{n}}\bigvee_{x-(x/u)}^x
(f_x^{\prime})du\\&\leq& \frac{\lambda (1+x)}{n+1}\sum_{k=1}
^{[\sqrt{n}]}\int_k^{k+1}\bigvee_{x-(x/u)}^x (f_x^{\prime})du\\&\leq&
\frac{\lambda (1+x)}{n+1}\sum_{k=1}^{[\sqrt{n}]}\bigvee_{x-(x/k)}^x (f_x^{\prime}).
\end{eqnarray*}
Thus we obtain
\begin{eqnarray}\label{p6}
|U_n(f^{\prime},x)|\leq \frac{\lambda(1+x)}{n+1}\sum_{k=1}^{[\sqrt{n}]}
\bigvee_{x-(x/k)}^x(f_x^{\prime})+\frac{x}{\sqrt{n}}\bigvee_{x-(x/\sqrt{n})}
^x(f_x^{\prime}).
\end{eqnarray}
Also
\begin{eqnarray*}
|V_n(f^{\prime},x)|&=&\bigg|\int_x^{\infty}\bigg(\int_x^t f_x^{\prime}(u)du\bigg)
J_n^a(x,t)dt\bigg|\\&=&\bigg|\int_z^{\infty}\bigg(\int_x^tf_x^{\prime}(u)
du\bigg)d_t(1-\alpha_n^a(x,t))+\int_x^{z}\bigg(\int_x^t f_x^{\prime}(u)du\bigg)d_t
(1-\alpha_n^a(x,t))\bigg|\\ &=&\bigg|\int_x^z f_x^{\prime}(u)
(1-\alpha_n^a(x,z))du-\int_x^z f_x^{\prime}(t)(1-\alpha_n ^a(x,t))dt\\
&&+\bigg(\int_x^t f_x^{\prime}(u)(1-\alpha_n^a(x,t))du\bigg)_z^{\infty}
-\int_z^{\infty} f_x^{\prime}(t)(1-\alpha_n^a(x,t))dt\bigg|\\&\leq&\bigg|\int_x^z
f_x^{\prime}(t)(1-\alpha_n^a(x,t))dt\bigg|+\bigg|\int_z^{\infty}f_x^{\prime}(t)
(1-\alpha_n^a(x,t))dt\bigg|.
\end{eqnarray*}
By using Lemma \ref{lm5}, with $z=x+(x/\sqrt{n}),$ we obtain
\begin{eqnarray*}
|V_n(f^{\prime},x)|&\leq&\frac{\lambda x(1+x)}{n+1}\int_z^{\infty}\bigvee_x^t
(f_x^{\prime})(t-x)^{-2}dt +\int_x^z \bigvee_x^t(f_x^{\prime})dt\\
&\leq&\frac{\lambda x(1+x)}{n+1}\int_z^{\infty}\bigvee_x^t(f_x^{\prime})
(t-x)^{-2}dt+\frac{x}{\sqrt{n}} \bigvee_x^{x+(x/\sqrt{n})}(f_x^{\prime}).
\end{eqnarray*}
By substitution of $v=\frac{x}{t-x},$ we get
\begin{eqnarray*}
\frac{\lambda x(1+x)}{n+1}\int_{x+(x/\sqrt{n})}^{\infty}\bigvee_x^t(f_x^{\prime})
(t-x)^{-2}dt&=& \frac{\lambda (1+x)}{n+1}\int_{0}^{\sqrt{n}}\bigvee_x^{x+(x/v)}
(f_x^{\prime})dv\\&\leq& \frac{\lambda (1+x)}{n+1}\sum_{k=0}^{[\sqrt{n}]}\int_k^{k+1}
\bigvee_x^{x+(x/v)}(f_x^{\prime})dv\\&\leq& \frac{\lambda (1+x)}{n+1}\sum_{k=0}
^{[\sqrt{n}]}\bigvee_x^{x+(x/k)}(f_x^{\prime}).
\end{eqnarray*}
Thus, we obtain
\begin{eqnarray}\label{p7}
|V_n(f^{\prime},x)|&\leq&\frac{\lambda (1+x)}{n+1}\sum_{k=0}
^{[\sqrt{n}]}\bigvee_x^{x+(x/k)}(f_x^{\prime})+\frac{x}{\sqrt{n}}
\bigvee_x^{x+(x/\sqrt{n})}(f_x^{\prime}).
\end{eqnarray}
From (\ref{p5})-(\ref{p7}), we get the required result.
\end{proof}

{\bf Acknowledgements}
The authors are extremely grateful to the reviewers for a careful reading of the manuscript
and making helpful suggestions leading to a better presentation of the paper. The second
author is thankful to the "Council of Scientific and Industrial Research"
India for financial support to carry out the above research work.

\end{document}